\documentclass[preprint,12pt]{elsarticle2}

\usepackage{amssymb}
\usepackage{amsthm}

\usepackage{color}

\usepackage[all]{xy}
\usepackage{graphicx,array,epic}

\usepackage{amsmath,amsfonts}

\newtheorem{thm}{Theorem}
\newtheorem{lemma}[thm]{Lemma}
\newtheorem{proposition}[thm]{Proposition}
\newtheorem{corollary}[thm]{Corollary}

\newdefinition{definition}[thm]{Definition}
\newdefinition{remark}[thm]{Remark}
\newdefinition{example}[thm]{Example}
\newdefinition{algorithm}[thm]{Algorithm}

\newproof{pf}{Proof}

\begin{document}

\begin{frontmatter}

\title{Toric ideals with linear components: an algebraic interpretation of clustering the cells of a contingency table}

\author[ec,ec1]{Enrico Carlini\corref{cor1}}
\ead{enrico.carlini@polito.it,enrico.carlini@monash.edu}

\author[fr]{Fabio Rapallo}
\ead{fabio.rapallo@unipmn.it}

\cortext[cor1]{Corresponding author. Tel: +61 03 9905 4462, Fax:
+61 3 9905 4403}

\address[ec]{DISMA - Department of Mathematical Sciences, Politecnico di Torino,
Corso Duca degli Abruzzi, 24, 10129 TURIN, Italy }

\address[ec1]{School of Mathematical Sciences, Monash University, 3800 VIC, Australia}

\address[fr]{Department DISIT, Universit\`a del Piemonte Orientale,
Viale Teresa Michel, 11, 15121 ALESSANDRIA, Italy}

\begin{abstract}
{In this paper we show that the agglomeration of
rows or columns of a contingency table with a hierarchical
clustering algorithm yields statistical models defined through
toric ideals. In particular, starting from the classical
independence model, the agglomeration process adds a linear part
to the toric ideal generated by the $2 \times 2$ minors.}
\end{abstract}


\end{frontmatter}

\section{Introduction} \label{introsect}

This paper aims to provide a geometric understanding of the
clustering techniques for discrete data, and in particular for
contingency tables. A two-way contingency table is an integer
data table which collects the outcomes of two categorical random
variables, i.e., it is a rectangular table of non-negative integer
numbers. Contingency table are widely used in Statistics and the
study of statistical models for this kind of data structures is an
active research area. Especially when the contingency table is
large, a natural problem is to ask whether there is some ways for
simplifying the data structure losing as few information as
possible. Therefore, a natural issue in that context is the
problem of clustering the rows, respectively the column, of a contingency
table, i.e. to find groups, respectively  columns, with similar
behavior.

In the literature, several techniques are studied in order to
cluster the rows and columns of a contingency table. As the number
of rows and columns is usually moderately small, most of the
existing algorithms fall in the class of hierarchical algorithms.
Among the most relevant references, we cite here the paper
\cite{greenacre:88}, where the clustering is defined in connection
with Correspondence Analysis, and \cite{hirotsu:09}, where the
clustering is endowed with suitable probability distributions,
essentially derived from the chi-square and the Wishart densities.
Such techniques are based on the spectral decomposition of a
special matrix derived from the observed data. In terms of
statistical models, they implicitly assume an underlying
independence model, i.e., they first assume that the two random
variables are independent, and then find discrepancies between the
observed counts and the expected counts under independence. In the
papers cited above, the reader can find several real-data examples
and a thorough description of the clustering process.

Since the definition of the algorithms is made in terms of the
chi-squared distance inherited from the exploratory techniques,
and in particular from Correspondence Analysis, there is
a lack of study of the geometric description of the statistical models
underlying the agglomerative process.

Here we use the language and tools from Algebraic Statistics
in order to determine under which conditions the clustering has a
natural algebraic and geometric counterpart. A survey of Algebraic
Statistics can be found in the book
\cite{drton|sturmfels|sullivant:09}. The use of Algebraic
Statistics for understanding the geometric structure of the
statistical models has already been considered in other papers.
Indeed, the availability of Computer Algebra systems and efficient
packages for polynomial computation have led to the study of
complex models for contingency tables. For instance,
\cite{carlini|rapallo:11} defines weakened independence models
using suitable sets of minors, while \cite{haraetal:09} adds a new
condition on a subtable to the standard independence model. In
\cite{boccietal:10} a modification of independence model to encode
a special behavior of the diagonal cells of a square table is
presented, while \cite{rapallo:12} defines statistical models to
encode the notion of outliers and patterns of outliers in
contingency tables through a model-based approach.

We collect some basic facts about the analysis of contingency
tables within Algebraic Statistics in the next sections, focusing
especially on the representation of a log-linear model in terms of
toric ideals. We also prove that under the independence model, the
action of merging two rows or two columns in a single cluster
produce a new statistical model for smaller tables which falls
again into the class of independence models. To do that, we prove
and apply some results about toric ideals. In particular, we study
how the ideal of the toric variety is affected by special
properties of and operations on the design matrix. As a general
reference to basic Algebraic Geometry we suggest \cite{harris:92}
while for the basic ideas about toric ideals we refer to
\cite{kreuzer|robbiano:05}.

In this paper we show that: $(i)$ the algebraic and geometric
structure of the independence model is a natural framework for the
clustering of the rows and columns of a contingency table; $(ii)$
merging two columns or two rows of the table it is possible
to characterize the induced toric model; $(iii)$ merging a row
with a column does not produce an easily interpretable statistical
model.

The paper is structured as follows. In Section \ref{toric-sect} we
briefly recall the basic properties of toric ideals that we need.
In Section \ref{models-sect}, we describe how toric ideals appear
in the study of statistical models. Section \ref{cluster-sect} is
devoted to give the statistical motivation for our study which
arise from the clustering algorithms of the rows or columns of a
contingency table. In Section \ref{main-sect} we collect our main
results while Section \ref{further-sect} contains pointer to
further studies.

\section{Toric varieties and toric ideals} \label{toric-sect}

In this paper we will work with {\em toric varieties} and with
{\em toric ideals}, thus we recall some of the basic notions that
we will need.

Given a $k \times (t+1)$ matrix $A=(a_{i,j})$ with non-negative
integer entries, this can be used to define a map
$T:\mathbb{R}^{t+1}\longrightarrow\mathbb{R}^k$ in the following
way
\[
T(\zeta_0,\ldots,\zeta_t)=
\]
\[
(\zeta_0^{a_{1,1}}\zeta_1^{a_{1,2}}\cdot\ldots\cdot\zeta_t^{a_{1,t+1}},\zeta_0^{a_{2,1}}\zeta_1^{a_{2,2}}\cdot\ldots\cdot\zeta_t^{a_{2,t+1}},\ldots,\zeta_0^{a_{k,1}}\zeta_1^{a_{k,2}}\cdot\ldots\cdot\zeta_t^{a_{k,t+1}})
 .
\]

The image of the map $T$,
$X=T(\mathbb{R}^{t+1})\subset\mathbb{R}^k$, is the {\em toric
variety} associated to the matrix $A$. Thus, $T$ is a
parametrization of $X$ and from this we see that $X$ is an
irreducible algebraic variety.

As a toric variety $X$ is an algebraic variety, it is the zero
locus of a set of polynomial equations. It is possible to
associate to $X$ an ideal $\mathcal{I}(X)\subset
R=\mathbb{R}[p_1,\ldots,p_k]$ which is the ideal containing all
polynomials vanishing on $X$.

As the matrix $A$ determines the toric variety via the map $T$, we
will denote with $T(A)$ the toric variety $T(\mathbb{R}^{k+1})$.
The matrix also determines the ideal of the toric variety, and we
will denote with $\mathcal{I}(A)\subset R$ the ideal
$\mathcal{I}(T(\mathbb{R}^{t+1}))$.

There are different ways of producing the toric ideal associated
to a matrix. In particular, we will make use of the following
algebraic approach. Consider the ring homomorphism
\[
\phi: \mathbb{R}[p_1,\ldots,p_k] \longrightarrow \mathbb{R}[\zeta_0,\ldots,\zeta_t]
\]
defined by setting
\begin{equation} \label{toric-alg}
\phi(p_i)=\zeta_0^{a_{i,1}}\zeta_1^{a_{i,2}}\cdot\ldots\cdot\zeta_t^{a_{i,t+1}}
\end{equation}
for $i=1,\ldots,k$ and extending it by using the homomorphism
properties. With these notations we get
$\mathcal{I}(A)=\mbox{Ker}(\phi)$.

Special features of the matrix $A$ reflect in geometric properties
of the associated toric variety $X$. In particular we examine here
the special case when the matrix $A$ has some repeated rows. In
such case $X$ is a hyperplane section of a cone over another toric
variety, namely the toric variety defined by the matrix $\bar A$
obtained from $A$ by deleting each replicate of a given row. We
can express this geometric property in algebraic terms as follows:

\begin{lemma}\label{repeatedlemma}
Let $A$ be an $k\times (t+1)$ matrix of non-negative integers
having $s$ rows $i_1,\ldots,i_s$ equal to each other. Denote with
$\bar A$ the $(k-s+1)\times (t+1)$ matrix obtained by $A$ deleting
the $s-1$ rows $i_2,\ldots,i_s$. Then one has
\[
\mathcal{I}(A)=\mathcal{I}(\bar A)+{\mathcal J}  ,
\]
where ${\mathcal J}$ is the ideal generated by the degree one
binomials $p_{i_1}-p_{i_2},\ldots,p_{i_1}-p_{i_s}$.
\end{lemma}
\begin{proof}
It is enough to provide a proof in the case $s=1$, and we may
assume that the first and the second rows are equal, i.e. $i_1=1$
and $i_2=2$. Then, $T(A)=Y\cap H$ where $H$ is the hyperplane of
equation $p_1-p_2=0$ and $Y$ is the cone over $T(\bar A)$ of
vertex the point $(1,0,\ldots,0)$. These remarks yields
\[
\mathcal{I}(A)\supseteq\mathcal{I}(\bar A)+\langle p_1-p_2\rangle
 .
\]
To prove that equality holds, let $F(p_1,\ldots,p_k)\in
\mathcal{I}(A)$ and construct the polynomial
$G(p_2,\ldots,p_k)=F(p_2,p_2,\ldots,p_k)$ by setting $p_1=p_2$. It
is enough to show that $G \in\mathcal{I}(\bar A)$, i.e. it is
enough to show that $G$ vanishes on all the points of $T(\bar A)$.
Let $(q_1,\ldots,q_k)\in T(\bar A)$ and notice that, as $T(\bar
A)$ is a cone, $(x,q_2,\ldots,q_k)\in T(\bar A)$ for all values of
$x\in\mathbb{R}$.  In particular, $(q_2,q_2,\ldots,q_k)\in
T(A)=T(\bar A)\cap H$ and hence $G(q_2,\ldots,q_k)=0$ and the
proof is completed.
\end{proof}

\section{Models and toric models} \label{models-sect}

In Statistics, a two-way contingency table collects the outcomes
of two categorical random variables, say $V_1$ and $V_2$, on a
sample. If the first variable has $I$ levels, and the second
variable has $J$ levels, we conventionally denote the sample space
by the cartesian product $\{1, \ldots, , I\} \times \{ 1 , \ldots,
J\}$ and the contingency table has integer nonnegative entries
$(n_{i,j})_{i=1, \ldots, I,j=1, \ldots, J}$ where $n_{i,j}$ is the
count of the individuals with $V_1=i$ and $V_2=j$.

A probability distribution for an $I \times J$ contingency table
is an $I \times J$ table of probabilities $(p_{i,j})_{i=1, \ldots,
I,j=1, \ldots, J}$ in the probability simplex
\[
\Delta = \left\{ (p_{i,j})_{i=1, \ldots, I,j=1, \ldots, J} \ : \
p_{i,j} \geq 0 \ , \ \sum_{i,j} p_{i,j} = 1 \right\}  .
\]

A statistical model for an $I \times J$ table is a subset of
$\Delta$. A well known class of statistical models for contingency
tables is the class of log-linear models, see e.g.
\cite{bishop|fienberg|holland:07}, defined through linear
conditions on the log of the probabilities, and therefore
restricted to the interior of $\Delta$. More precisely, given an
integer matrix $A$ with dimension $IJ \times (t+1)$, the
log-linear model defined by $A$ is characterized by the linear
system
\begin{equation} \label{loglindef}
\log(p) = A \beta  .
\end{equation}
Here, $\log(p)$ is the vector of the log-probabilities (taken by
ordering the cells lexicographically), $\beta =\left( \beta_0,\ldots, \beta_t \right) \in {\mathbb
R}^{t+1}$ is the vector of model parameters. The number of free
parameters is equal to $\mathrm{rk}(A)$, the rank of $A$, and the
degrees of freedom of the model are the difference $IJ-\mathrm{
rk}(A)$.

Exponentiating Eq. \eqref{loglindef}, we obtain a different
expression of the same model, namely
\begin{equation} \label{toricdef}
p = \zeta^A  ,
\end{equation}
where $\zeta^A$ is written in vector notation, and $\zeta=
\exp(\beta)$ is the vector of nonnegative parameters. This model
is called toric model, as it expresses the probabilities in
monomial form, and it allows to extend the model on the
boundary of $\Delta$, see \cite{rapallo:07} and
\cite{drton|sturmfels|sullivant:09}. Notice that Eq.
\eqref{toricdef} is just Eq. \eqref{toric-alg} written in a
shorter notation.

Eliminating the $\zeta$ parameters from Eq. \eqref{toricdef} we
obtain that the toric model is the variety defined by the toric ideal
${\mathcal I}(A)$ introduced above:
\begin{equation}\label{toricideallatex}
{\mathcal I}(A) = \langle p^a- p^b \ : \ A^t a= A^tb \rangle
\end{equation}
and therefore the model is:
\begin{equation}
{\mathcal M}_A = T(A) \cap \Delta  .
\end{equation}
Thus, toric models are defined by a finite set of binomial
equations (up to the normalizing constant).

In this paper we will use especially the independence model,
defined by the model matrix
\[
A = [{\bf 1} , r_1, \ldots, r_I, c_1, \ldots, c_J]  ,
\]
where
\begin{itemize}
\item {\bf 1} is a column vectors of $1$s;

\item $r_i$ is the indicator vector of the cells in the $i$-th row of the table;

\item $c_j$ is the indicator vector of the cells in the $j$-th column of the table.
\end{itemize}
The independence model contains all probability matrices such that
the variables $V_1$ and $V_2$ are statistically independent. It is
known, see \cite{drton|sturmfels|sullivant:09} that the
corresponding toric ideal ${\mathcal I}(A)$ is generated by all $2
\times 2$ minors of the probability table, that is
\begin{equation} \label{toricind}
{\mathcal I}(A) = \langle p_{i,j}p_{h,k} - p_{i,k}p_{h,j} \ : \ 1
\leq i < h \leq I, \ 1 \leq j < k \leq J \rangle  .
\end{equation}

\section{Clustering algorithms and their geometric counterpart}
\label{cluster-sect}

Let $R_1, \ldots, R_I$ be the labels of the rows and $C_1, \ldots, C_J$ be the labels of the columns of the contingency table under study. The classical agglomerative hierarchical clustering algorithms for a contingency table proceed as follows:
\begin{itemize}
\item At the initial step, each label is a separate cluster:
\[
\{R_1\}, \ldots, \{R_I\},\{C_1\}, \ldots, \{C_J\}  .
\]
\item At step 1, we merge two elements of the partition above, based on some similarity criterion, and we obtain a partition with $(I+J-1)$ clusters. There are different ways for defining the similarity in this step, and such different definitions yield different clustering algorithms. In our geometric study, we do not need to choose a specific similarity criterion, since our analysis is valid in general.

\item At each step, we merge two elements of the partition, namely the two most similar elements, so that we lose a cluster at each step.

\item At step $(I+J-2)$ we define two clusters $G_1$ and $G_2$, and at the last step all row and column labels are merged together in a grand cluster.
\end{itemize}

Although in general an agglomerative clustering algorithm does not consider any restrictions on the intermediate steps, meaning that all pairs of clusters may be merged, in contingency tables analysis, the algorithms in the literature are usually defined to separately agglomerate the rows and the columns of the table, and we prove here that this prescription is fully motivated also under the geometric point of view. Therefore, we begin our analysis with the study of step 1 of the procedure above and, without loss of generality, we suppose that we have to merge the last two columns of the table.

Using the independence model, to merge the last two columns means
to constrain the equality of the last two parameters: $\beta_{t-1} =
\beta_{t}$ in the log-linear representation \eqref{loglindef} or
$\zeta_{t-1} = \zeta_{t}$ in the toric representation
\eqref{toricdef}.

In a natural way, these equalities define a sub-model obtained
from the matrix $A$ of the independence model, with a new model
matrix $\tilde A$ obtained by summing the last two columns
of $A$.

The scheme is

\begin{displaymath}
    \xymatrix{
        A=[{\bf 1},a_1, \ldots ,a_t] \ar[r]\ar[d] & \mathcal{I}(A)\ar[d] \\
        \tilde A=[{\bf 1},a_1, \ldots ,a_{t-2},a_{t-1}+a_{t}] \ar[r]& \mathcal{I}({\tilde A})}
\end{displaymath}

The relation between these two ideals is given in Theorem \ref{diagramprop}, namely
\begin{equation}\label{IAeItildaA}
\mathcal{I}(\tilde A)= \mathcal{I}(A)+{\mathcal J},
\end{equation}
where ${\mathcal J}$ is a suitable ideal.

\begin{remark}
In this paper we assume that the first column of $A$ is ${\bf
1}=(1, \ldots, 1)^T$. This a standard assumption for log-linear
models and it does not affect the generality of our analysis. The
algebraic counterpart of this assumption is the fact that the
toric ideals involved are homogeneous, in particular it is
possible to generate them using homogeneous binomials.
\end{remark}

In terms of binomials, it is important to know whether the toric
model associated to $\tilde A$ satisfies linear equations.
Especially when this equations come from the ideal $\mathcal{J}$
in \eqref{IAeItildaA}, this means that the linear binomials are
generated by the clustering process. The relevance of degree one
elements of the toric ideal rest on the identification of two
cells of the table as follows.

\begin{remark}
The identification (or clustering) of two cells, say $(i,j)$ and $(h,k)$,
has its algebraic counterpart in the linear binomial
\[
p_{i,j} - p_{h,k}  .
\]
In fact, the equation $p_{i,j} - p_{h,k} = 0$ adds a constraint
and the probabilities of the cells $k$ and $h$ must be equal.
\end{remark}

We now describe what is the effect on the toric ideals of summing
up two columns of the matrix $A$. In the examples below the
computations are carried out with the software {\tt 4ti2}, see
\cite{4ti2}.

\begin{example}\label{ex1}
For instance the following matrix
is the model matrix of the independence model for $3 \times 3$
contingency tables, and represents the most simple example for our
theory.
\begin{equation} \label{mat3x3}
A=\left(\begin{array}{ccccccc}
1 & 1 & 0 & 0 & 1 & 0 & 0 \\
1 & 1 & 0 & 0 & 0 & 1 & 0 \\
1 & 1 & 0 & 0 & 0 & 0 & 1 \\
1 & 0 & 1 & 0 & 1 & 0 & 0 \\
1 & 0 & 1 & 0 & 0 & 1 & 0 \\
1 & 0 & 1 & 0 & 0 & 0 & 1 \\
1 & 0 & 0 & 1 & 1 & 0 & 0 \\
1 & 0 & 0 & 1 & 0 & 1 & 0 \\
1 & 0 & 0 & 1 & 0 & 0 & 1
\end{array}\right)
\end{equation}
and the corresponding toric ideal is
\begin{equation}
\begin{split}
\mathcal{I}(A) = \langle -p_{1,3}p_{2,1} +p_{1,1}p_{2,3},
p_{1,3}p_{3,2} -p_{1,2}p_{3,3}, p_{2,3}p_{3,2} -p_{2,2}p_{3,3}, \\
p_{2,3}p_{3,1} -p_{2,1}p_{3,3}, -p_{1,2}p_{2,1} +p_{1,1}p_{2,2},
p_{1,3}p_{2,2} -p_{1,2}p_{2,3}, p_{2,2}p_{3,1} -p_{2,1}p_{3,2}, \\
p_{1,3}p_{3,1} -p_{1,1}p_{3,3}, -p_{1,2}p_{3,1} +p_{1,1}p_{3,2}
\rangle  .
\end{split}
\end{equation}
If we merge the last two columns of the contingency table, we obtain a model matrix
$\tilde A$ with some repeated rows, and the
corresponding toric ideal is
\begin{equation}
\begin{split}
\mathcal{I}(\tilde A) = \langle p_{3,2} -p_{3,3}, -p_{1,2}
+p_{1,3}, -p_{2,2} +p_{2,3}, p_{1,2}p_{3,1} -p_{1,1}p_{3,3}, \\
p_{2,2}p_{3,1} -p_{2,1}p_{3,3}, -p_{1,2}p_{2,1} +p_{1,1}p_{2,2}
\rangle  .
\end{split}
\end{equation}
The new ideal $\mathcal{I}(\tilde A)$ contains ${\mathcal I}(A)$
and has three additional generators: $p_{3,2} -p_{3,3}$, $-p_{1,2}
+p_{1,3}$, and $-p_{2,2} +p_{2,3}$. This means that the analytical
condition $\beta_5=\beta_6$ in the log-linear formulation of the
independence model translates into the identification of the cells
in the last two columns of the $3 \times 3$ table: $(1,2)$ with
$(1,3)$ in the first row, $(2,2)$ with $(2,3)$ in the second row,
and  $(3,2)$ with $(3,3)$ in the third row. A similar argument can be repeated when summing any two columns which are both rows indicators, or columns indicators of the table.
\end{example}

\begin{remark}
Through a quick inspection of the generators of ${\mathcal I}(A)$ in the example above, one sees that the ideal is generated by:
\begin{itemize}
\item[(a)] three binomials of degree two: $p_{1,2}p_{3,1} -p_{1,1}p_{3,3},
p_{2,2}p_{3,1} -p_{2,1}p_{3,3}, -p_{1,2}p_{2,1} +p_{1,1}p_{2,2}$. Such binomials are the generators of the toric ideal of the independence model for a $3 \times 2$ table;

\item[(b)] three binomials of degree one: $p_{3,2} -p_{3,3}, -p_{1,2}
+p_{1,3}, -p_{2,2} +p_{2,3}$. Such binomial yields the identification of the two merged columns, as discussed above.
\end{itemize}
The agglomeration of the last two columns of the table is sketched in Figure \ref{fig-merge}.
In the next section, we will prove that this fact is valid in general for the independence model.

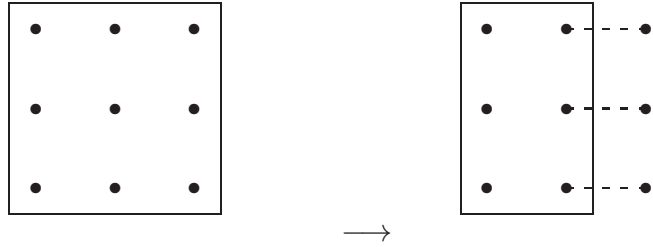
\begin{figure}
\begin{center}
\begin{tabular}{ccc}
\begin{picture}(100,100)(0,0)
\put(10,10){\line(0,1){80}}
\put(10,10){\line(1,0){80}}
\put(10,90){\line(1,0){80}}
\put(90,10){\line(0,1){80}}

\put(20,20){\circle*{4}}
\put(20,50){\circle*{4}}
\put(20,80){\circle*{4}}

\put(50,20){\circle*{4}}
\put(50,50){\circle*{4}}
\put(50,80){\circle*{4}}

\put(80,20){\circle*{4}}
\put(80,50){\circle*{4}}
\put(80,80){\circle*{4}}

\end{picture}

& \qquad $\longrightarrow$ \qquad &

\begin{picture}(100,100)(0,0)
\put(10,10){\line(0,1){80}}
\put(10,10){\line(1,0){50}}
\put(10,90){\line(1,0){50}}
\put(60,10){\line(0,1){80}}

\put(20,20){\circle*{4}}
\put(20,50){\circle*{4}}
\put(20,80){\circle*{4}}

\put(50,20){\circle*{4}}
\put(50,50){\circle*{4}}
\put(50,80){\circle*{4}}

\put(80,20){\circle*{4}}
\put(80,50){\circle*{4}}
\put(80,80){\circle*{4}}

\put(50,20){\dashline[8]{3}(0,0)(30,0)}
\put(50,50){\dashline[8]{3}(0,0)(30,0)}
\put(50,80){\dashline[8]{3}(0,0)(30,0)}
\end{picture}
\end{tabular}
\end{center}
\caption[]{The agglomeration of columns $2$ and $3$ in a $3 \times 3$ table described in Example \ref{ex2}. The boxes represent the regions where the independence model holds, and the lines represent the identified cells.} \label{fig-merge}
\end{figure}
\end{remark}

The behavior exhibited in Example \ref{ex1} is very peculiar and, in general, the toric ideal is not so easily described. The next example is not directly related to the independence model, but it shows that linear polynomials are not
enough in general for describing the toric ideal ${\mathcal
I}(\tilde A)$. Two more examples shows that linear
polynomials does not appear when merging a row with a column of
the table.

\begin{example}\label{ex2} In general, it is not enough to add linear binomials to ${\mathcal I}(A)$ to obtain ${\mathcal I}(\tilde A)$. For instance, if we compute the toric ideal associated to the following matrix
\[
A=\left(\begin{array}{ccccc}
1 & 0 & 1 & 0 & 0 \\
0 & 1 & 0 & 1 & 0 \\
1 & 1 & 0 & 0 & 0 \\
0 & 0 & 1 & 0 & 1
\end{array}\right)
\]
in the ring ${\mathbb R}[p_1,p_2,p_3,p_4]$, we obtain the zero
ideal, i.e. $T(A)=\mathbb{R}^4$ but if we sum the last two columns of $A$, we get the
matrix $\tilde A$ and $\mathcal{I}({\tilde A}) = \mathcal{I}(A) + (p_1p_2-p_3p_4)$. Thus
a generator of degree two appears in $\mathcal{I}({\tilde A})$.
\end{example}

\begin{example}
Now we sum two columns of the
model matrix of the independence model pertaining to different objects, i.e., one indicator
vector of a row and one indicator vector of a column. Let us
consider Example \ref{ex1} on the $3 \times 3$ independence
model, but now we sum the indicator vectors of the last row and of
the last column. The new matrix is now
\begin{equation} \label{tildeb}
\tilde B=\left(\begin{array}{ccccccc}
1 & 1 & 0 & 0 & 1 & 0  \\
1 & 1 & 0 & 0 & 0 & 1  \\
1 & 1 & 0 & 1 & 0 & 0  \\
1 & 0 & 1 & 0 & 1 & 0  \\
1 & 0 & 1 & 0 & 0 & 1  \\
1 & 0 & 1 & 1 & 0 & 0  \\
1 & 0 & 0 & 1 & 1 & 0  \\
1 & 0 & 0 & 1 & 0 & 1  \\
1 & 0 & 0 & 2 & 0 & 0
\end{array}\right)  .
\end{equation}
As the matrices $A$ and $\tilde B$ have the same kernel, the toric
ideals must be equal: ${\mathcal I}(A) = {\mathcal I}(\tilde B)$.
From the point of view of Statistics, this behavior is natural,
since no degrees of freedom are added when moving from $A$ to
$\tilde B$.
\end{example}

\begin{example}
When the agglomeration process is iterated twice the situation is more complicated. Let us consider the independence model
for $4 \times 4$ contingency tables, whose model matrix $A$
written as in Eq. \eqref{mat3x3} with the suitable number of rows
and columns. The matrix $A$ has $16$ rows and $9$ columns: $1$
column of $1$s, $4$ columns with the indicator of the rows, and
$4$ columns with the indicators of the columns. The toric ideal
${\mathcal I}(A)$ is generated by the $36$ $2 \times 2$ minors of
the probability table, as prescribed from Eq. \eqref{toricind}.
When we merge the last row with the last column of the table
(i.e., we sum the fifth and the ninth column of $A$), we obtain a
matrix $\tilde B$ in analogy with the matrix in Eq. \eqref{tildeb}
and the corresponding ideal is ${\mathcal I}(\tilde B) = {\mathcal
I}(A)$. If we merge two further objects, namely the third row with
the third column of the table, we define a new matrix $\tilde C$,
and the ideal is
\[
{\mathcal I}(\tilde C) = {\mathcal I}(A) + \langle p_{4,3}-p_{3,4} , p_{3,4}^2 - p_{3,3}p_{4,4} \rangle
\]
and the two added binomials do not have a simple statistical counterpart.
\end{example}

\section{General facts and application to the independence model}
\label{main-sect}

There is a quite general relation between the toric ideal of $A$ and the one of $\tilde A$.

\begin{thm}\label{diagramprop} Consider the diagram of commutative rings
\begin{displaymath}
    \xymatrix{
        k[p_1,\ldots,p_{IJ}] \ar[d]_\psi \ar[r]_\phi & k[\zeta_0,\ldots,\zeta_{t-1} ,\zeta_{t}]\ar[dl]_\alpha \\
        k[\zeta_0,\ldots,\zeta_{t-1}]                       & }
\end{displaymath}

where the maps are defined as follows:
\[\phi(p_i)=\zeta_0^{1}\zeta_1^{a_{i,1}}\cdot\ldots\cdot\zeta_t^{a_{i,t}},\]
\[\psi(p_i)=\zeta_0^{1}\zeta_1^{a_{i,1}}\cdot\ldots\cdot\zeta_{t-1}^{a_{i,t}+a_{i,t}}  ,\]
and $\alpha(\zeta_i)=\zeta_i$ for $i\neq t$ and  $\alpha(\zeta_t)=\zeta_{t-1}$.
Then, with the notation of the previous sections, we have
\[\mathcal{I}(\tilde A)=\mathcal{I}(A)+\alpha^{-1}(\zeta_{t-1}-\zeta_t).\]
\end{thm}
\begin{proof} The proof follows immediately noticing that $\phi=\alpha\circ\psi$.
\end{proof}

We now derive some interesting facts about the degree one elements of $\mathcal I(\tilde A)$.

\begin{proposition}
A linear binomial $p_k-p_h$ belongs to $\mathcal I{\tilde A}$ if
and only if $a_r(k)=a_r(h)$ for $r=0, \ldots, t-2$ and
$a_{t-1}(k)+a_{t}(k)=a_{t-1}(h)+a_{t}(h)$.
\end{proposition}
\begin{proof}
We directly use the parametrization to compute
\[p_h-p_k=\tilde\zeta_0 \tilde\zeta_1^{a_1(h)} \ldots \tilde\zeta_{t-2}^{a_{t-2}(h)}\tilde\zeta_{t-1}^{a_{t-1}(h)+a_{t}(h)}-\tilde\zeta_0 \tilde\zeta_1^{a_1(k)} \ldots \tilde\zeta_{t-2}^{a_{t-2}(k)}\tilde\zeta_{t-1}^{a_{t-1}(k)+a_{t}(k)}\]
and this is the zero polynomial if and only if the statement holds.
\end{proof}

In several situation of practical interest, the model matrix is a binary matrix, i.e. it only involves zeros and ones. This is the case, for example, for the independence model.

\begin{corollary}
Let $A$ and ${\tilde A}$ as above be two binary matrices. A linear
polynomial $p_k-p_h$ belongs to ${\mathcal I}(\tilde A)$ if and
only if it belongs to ${\mathcal I}(A)$ or $a_r(k)=a_r(h)$ for
$r=0, \ldots, t-2$ and $a_{t-1}(k) \ne a_{t}(k), a_{t-1}(h) \ne
a_{t}(h)$.
\end{corollary}
\begin{proof}
Let $k$ and $h$ be two rows of $A$ such that $a_r(k)=a_r(h)$ for
$r=0, \ldots, t-2$. Let us consider the $2 \times 2$ subtable of
$A$ formed by the rows $k$ and $h$, and by the columns $t-1$ and
$t$. It is easy to check that there are only four possible
configurations, up to permutations of rows and columns:
\begin{center}
\begin{tabular}{ccccccc}
$\begin{pmatrix}
1 & 0 \\
0 & 1
\end{pmatrix}$
& \qquad &
$\begin{pmatrix}
1 & 0 \\
1 & 0
\end{pmatrix}$
& \qquad &
$\begin{pmatrix}
1 & 0 \\
0 & 0
\end{pmatrix}$
& \qquad &
$\begin{pmatrix}
0 & 0 \\
0 & 0
\end{pmatrix}$ \\
$(a)$ & & $(b)$ & & $(c)$ & & $(d)$
\end{tabular}
\end{center}
Among such configurations, only the first one satisfies the last
condition of the statement. In configuration $(a)$, the new
binomial $p_k-p_h$ is added to ${\mathcal I}(\tilde A)$ as a
consequence of Lemma \ref{repeatedlemma}, while in
configurations $(b)$ and $(d)$ the binomial $p_k-p_h$ is already
in ${\mathcal I}(A)$ because the rows $k$-th and $h$-th of $A$ are
equal. Configuration $(c)$ does not produce any linear
binomial since the two rows in ${\tilde A}$ are not equal.
\end{proof}

We conclude with a nice result about the independence model. This
explain why Example \ref{ex1} is so well-behaved.

\begin{proposition}\label{mainres}
Let $A$ be the model matrix of an independence model. Produce a
matrix $\tilde A$ by summing up two columns of $A$ both rows
indicators, or both columns indicators, of the table. Then,
$$\mathcal{I}(\tilde A)=\mathcal{I}(A)+\mathcal{J}$$ where the ideal
$\mathcal{J}$ is generated by linear binomials. More precisely, if
we sum the indicator functions of columns $a$ and $b$, then
$\mathcal{J}$ contains all the binomials $p_{i,a}-p_{i,b}$ for all
values of $i$.
\end{proposition}
\begin{proof}
Without loss of generality let us assume that $A$ is the model
matrix of an independence model over a $I\times J$ table $T_1$.
Also, we assume to sum the last two columns of $A$, say column
$t-1$ and columns $t$, coming from the indicator functions of the
last two columns of the table, say column $J-1$ and column $J$.
Notice that $\tilde A$ has rows which are equal in pairs, namely
the rows indexed by
\[
p_{1,J-1},\ldots,p_{I,J-1}
\]
coincide, respectively, with the rows indexed by,
\begin{equation}\label{multiplerows}
p_{1,J},\ldots,p_{I,J}  .
\end{equation}
Thus, we construct the matrix $\bar B$ obtained by $\tilde A$
deleting all the rows in \eqref{multiplerows}. Using Lemma
\ref{repeatedlemma}, we know that
\[
\mathcal{J}(\bar B)+\langle
p_{1,J-1}-p_{1,J},\ldots,p_{I,J-1}-p_{I,J}\rangle=\mathcal{J}({\tilde
A})  .
\]

Now, we notice that $\bar B$ is the model matrix of an
independence model on the table $T_2$ obtained by removing column
$J$ by table $T_1$. In particular, $\mathcal{J}(\bar B)$ is
generated by the $2\times 2$ minors of $T_2$, i.e. by the $2\times
2$ minors of $T_1$ not involving the last column, i.e. by the
$2\times 2$ minors not involving any of the variables
$p_{1,J},\ldots,p_{I,J}$. Thus, it is immediate to check that
\[
\mathcal{J}(A)+\langle
p_{1,J-1}-p_{1,J},\ldots,p_{I,J-1}-p_{I,J}\rangle=\mathcal{J}(\bar
B)+\langle p_{1,J-1}-p_{1,J},\ldots,p_{I,J-1}-p_{I,J}\rangle  ,
\]
and the conclusion follows.
\end{proof}

\begin{remark} Proposition \ref{mainres} is the key for iterating the clustering process within the independence model. Identifying two columns of the table produces linear equations given by the identified cells and a smaller table. If we now identify columns, or rows, of this new table, we have to add new linear equations and we have to deal with new, smaller tables.

\end{remark}

\section{Further questions} \label{further-sect}

In this paper we addressed the following question:

\begin{quote}
\item[$Q_1$] What does it happen when we identify two parameters
of the log-linear (or toric) model? Recall that the first
parameter $\beta_0$ is fixed, so that the question only concerns
$\beta_1, \ldots, \beta_t$?
\end{quote}

We provided examples and general results, see Theorem
\ref{diagramprop}, about this question.  Also, we answered
question $Q_1$ in some special cases, e.g. for the independence
model we produced a complete answer, see Proposition
\ref{mainres}. Question $Q_1$ leads to new interesting problems to
be addressed, both from the point of view of Algebra and from the
point of view of applications to Statistics.

First, it is now natural to consider new questions, which is in
some sense the converse of $Q_1$, namely:
\begin{quote}
\item[$Q_2$]
Given a toric model with model matrix $A$, what does it happen to
$\mathcal{I}(A)$ (and to the model matrix $A$) when we add a linear
binomial of the form $p_k-p_h$ to the generators of $\mathcal{I}(A)$?
\end{quote}

The treatment of this question will be subject to further
analysis. We think that the following is a promising approach:
first we form the toric ideal associated to the matrix $A$.
\[
A=[a_0, \ldots ,a_t]  \mapsto {\mathcal I}(A)  .
\]
Then we form the ideal $\mathcal{J}={\mathcal I}(A)+\langle
p_k-p_h \rangle$ and we consider the following questions: under
which conditions on the linear binomial $p_h-p_k$ is
$\mathcal{J}=\mathcal{I}(B)$ for some matrix $B$? If, for a given
choice of a linear binomial, $\mathcal{J}$ is not a toric ideal,
how can we find a matrix $B$ such that
$\mathcal{J}\subset\mathcal{I}(B)$?

From the perspective of the analysis of statistical models, it is
interesting to study other approaches to cluster analysis for
contingency tables. In this paper, we focused our attention to the
hierarchical agglomerative algorithms, but there are also
different approaches to the problem. For instance, in
\cite{govaert|nadif:07} the problem of clustering is considered in the framework of mixture models. Moreover, also with the
names of bi-clustering or co-clustering, several algorithms are
now available for a wide range of applications, from Psychometry
to Molecular Biology. For more details on such techniques see for
instance \cite{ritschard|zighed:03}, \cite{madeira|oliveira:04},
and \cite{govaert|nadif:10}. Such procedures for finding patterns
in data matrices differ in several aspects (the patterns they
seek, the assumptions on the underlying probability models, and so
on), but all of them are essentially based on mixtures of
independence models or other log-linear models with support on
suitable regions of the contingency table to be determined by
minimizing some cost functions. The geometry of such statistical
models is strictly related with the notion of nonnegative rank of
a matrix, and some preliminary results can be found in
\cite{cohen|rothblum:93}, \cite{carlini|rapallo:10} and
\cite{boccietal:11}, where some applications to mixture models in
statistics are introduced.

\section*{Acknowledgements}
F. Rapallo is partially supported by the Italian Ministry for
University and Research, programme PRIN2009, grant number
2009H8WPX5.

\bibliographystyle{elsarticle-num}
\bibliography{carlinirapallo}







\end{document}